\documentclass[a4paper,11pt]{amsart}    
\usepackage{latexsym, amsmath, amsthm, amssymb, setspace, verbatim}
\usepackage[all]{xy}
\usepackage[utf8]{inputenc} \usepackage[T1]{fontenc} \usepackage{lmodern}
\usepackage{hyperref}
\usepackage{enumitem}

\title{ Strongly self-absorbing \cstar-dynamical systems, III }
\author{ Gábor Szabó }
\address{Fraser Noble Building, Institute of Mathematics, University of Aberdeen, \linebreak \text{}\hspace{2.8mm} Aberdeen AB24 3UE, Scotland, UK}
\email{gabor.szabo@abdn.ac.uk}
\thanks{\emph{Supported by:} SFB 878 \emph{Groups, Geometry and Actions} and EPSRC grant EP/N00874X/1}

\numberwithin{equation}{section}

\begin{document}

\renewcommand\matrix[1]{\left(\begin{array}{*{10}{c}} #1 \end{array}\right)}  
\newcommand\set[1]{\left\{#1\right\}}  

\newcommand{\IA}[0]{\mathbb{A}} \newcommand{\IB}[0]{\mathbb{B}}
\newcommand{\IC}[0]{\mathbb{C}} \newcommand{\ID}[0]{\mathbb{D}}
\newcommand{\IE}[0]{\mathbb{E}} \newcommand{\IF}[0]{\mathbb{F}}
\newcommand{\IG}[0]{\mathbb{G}} \newcommand{\IH}[0]{\mathbb{H}}
\newcommand{\II}[0]{\mathbb{I}} \renewcommand{\IJ}[0]{\mathbb{J}}
\newcommand{\IK}[0]{\mathbb{K}} \newcommand{\IL}[0]{\mathbb{L}}
\newcommand{\IM}[0]{\mathbb{M}} \newcommand{\IN}[0]{\mathbb{N}}
\newcommand{\IO}[0]{\mathbb{O}} \newcommand{\IP}[0]{\mathbb{P}}
\newcommand{\IQ}[0]{\mathbb{Q}} \newcommand{\IR}[0]{\mathbb{R}}
\newcommand{\IS}[0]{\mathbb{S}} \newcommand{\IT}[0]{\mathbb{T}}
\newcommand{\IU}[0]{\mathbb{U}} \newcommand{\IV}[0]{\mathbb{V}}
\newcommand{\IW}[0]{\mathbb{W}} \newcommand{\IX}[0]{\mathbb{X}}
\newcommand{\IY}[0]{\mathbb{Y}} \newcommand{\IZ}[0]{\mathbb{Z}}

\newcommand{\CA}[0]{\mathcal{A}} \newcommand{\CB}[0]{\mathcal{B}}
\newcommand{\CC}[0]{\mathcal{C}} \newcommand{\CD}[0]{\mathcal{D}}
\newcommand{\CE}[0]{\mathcal{E}} \newcommand{\CF}[0]{\mathcal{F}}
\newcommand{\CG}[0]{\mathcal{G}} \newcommand{\CH}[0]{\mathcal{H}}
\newcommand{\CI}[0]{\mathcal{I}} \newcommand{\CJ}[0]{\mathcal{J}}
\newcommand{\CK}[0]{\mathcal{K}} \newcommand{\CL}[0]{\mathcal{L}}
\newcommand{\CM}[0]{\mathcal{M}} \newcommand{\CN}[0]{\mathcal{N}}
\newcommand{\CO}[0]{\mathcal{O}} \newcommand{\CP}[0]{\mathcal{P}}
\newcommand{\CQ}[0]{\mathcal{Q}} \newcommand{\CR}[0]{\mathcal{R}}
\newcommand{\CS}[0]{\mathcal{S}} \newcommand{\CT}[0]{\mathcal{T}}
\newcommand{\CU}[0]{\mathcal{U}} \newcommand{\CV}[0]{\mathcal{V}}
\newcommand{\CW}[0]{\mathcal{W}} \newcommand{\CX}[0]{\mathcal{X}}
\newcommand{\CY}[0]{\mathcal{Y}} \newcommand{\CZ}[0]{\mathcal{Z}}

\newcommand{\FA}[0]{\mathfrak{A}} \newcommand{\FB}[0]{\mathfrak{B}}
\newcommand{\FC}[0]{\mathfrak{C}} \newcommand{\FD}[0]{\mathfrak{D}}
\newcommand{\FE}[0]{\mathfrak{E}} \newcommand{\FF}[0]{\mathfrak{F}}
\newcommand{\FG}[0]{\mathfrak{G}} \newcommand{\FH}[0]{\mathfrak{H}}
\newcommand{\FI}[0]{\mathfrak{I}} \newcommand{\FJ}[0]{\mathfrak{J}}
\newcommand{\FK}[0]{\mathfrak{K}} \newcommand{\FL}[0]{\mathfrak{L}}
\newcommand{\FM}[0]{\mathfrak{M}} \newcommand{\FN}[0]{\mathfrak{N}}
\newcommand{\FO}[0]{\mathfrak{O}} \newcommand{\FP}[0]{\mathfrak{P}}
\newcommand{\FQ}[0]{\mathfrak{Q}} \newcommand{\FR}[0]{\mathfrak{R}}
\newcommand{\FS}[0]{\mathfrak{S}} \newcommand{\FT}[0]{\mathfrak{T}}
\newcommand{\FU}[0]{\mathfrak{U}} \newcommand{\FV}[0]{\mathfrak{V}}
\newcommand{\FW}[0]{\mathfrak{W}} \newcommand{\FX}[0]{\mathfrak{X}}
\newcommand{\FY}[0]{\mathfrak{Y}} \newcommand{\FZ}[0]{\mathfrak{Z}}

\renewcommand{\phi}[0]{\varphi}
\newcommand{\eps}[0]{\varepsilon}

\newcommand{\id}[0]{\operatorname{id}}		
\newcommand{\eins}[0]{\mathbf{1}}			
\newcommand{\ad}[0]{\operatorname{Ad}}
\newcommand{\ev}[0]{\operatorname{ev}}
\newcommand{\fin}[0]{{\subset\!\!\!\subset}}
\newcommand{\diam}[0]{\operatorname{diam}}
\newcommand{\Hom}[0]{\operatorname{Hom}}
\newcommand{\Aut}[0]{\operatorname{Aut}}
\newcommand{\dimrok}[0]{\dim_{\mathrm{Rok}}}
\newcommand{\dimrokc}[0]{\dim_{\mathrm{Rok}}^{\mathrm{c}}}
\newcommand{\dimrokeins}[0]{\dimrok^{\!+1}}
\newcommand*\onto{\ensuremath{\joinrel\relbar\joinrel\twoheadrightarrow}} 
\newcommand*\into{\ensuremath{\lhook\joinrel\relbar\joinrel\rightarrow}}  
\newcommand{\dst}[0]{\displaystyle}
\newcommand{\cstar}[0]{\ensuremath{\mathrm{C}^*}}
\newcommand{\dist}[0]{\operatorname{dist}}
\newcommand{\ue}[1]{{~\approx_{\mathrm{u},#1}}~}
\newcommand{\End}[0]{\operatorname{End}}
\newcommand{\ann}[0]{\operatorname{Ann}}
\newcommand{\strict}[0]{\stackrel{\text{\tiny str}}{\longrightarrow}}
\newcommand{\cc}[0]{\simeq_{\mathrm{cc}}}
\newcommand{\scc}[0]{\simeq_{\mathrm{scc}}}
\newcommand{\vscc}[0]{\simeq_{\mathrm{vscc}}}

\newcommand{\cf}[2]{cf.\ {\cite[#2]{#1}}}
\renewcommand{\see}[2]{see {\cite[#2]{#1}}}

\newtheorem{satz}{Satz}[section]		
\newtheorem{cor}[satz]{Corollary}
\newtheorem{lemma}[satz]{Lemma}
\newtheorem{prop}[satz]{Proposition}
\newtheorem{theorem}[satz]{Theorem}
\newtheorem*{theoreme}{Theorem}

\theoremstyle{definition}
\newtheorem{conjecture}[satz]{Conjecture}
\newtheorem*{conjecturee}{Conjecture}
\newtheorem{defi}[satz]{Definition}
\newtheorem*{defie}{Definition}
\newtheorem{defprop}[satz]{Definition \& Proposition}
\newtheorem{nota}[satz]{Notation}
\newtheorem*{notae}{Notation}
\newtheorem{rem}[satz]{Remark}
\newtheorem*{reme}{Remark}
\newtheorem{example}[satz]{Example}
\newtheorem{defnot}[satz]{Definition \& Notation}
\newtheorem{question}[satz]{Question}
\newtheorem*{questione}{Question}

\newenvironment{bew}{\begin{proof}[Proof]}{\end{proof}}


\begin{abstract} 
In this paper, we accomplish two objectives. Firstly, we extend and improve some results in the theory of (semi-)strongly self-absorbing \cstar-dynamical systems, which was introduced and studied in previous work.
In particular, this concerns the theory when restricted to the case where all the semi-strongly self-absorbing actions are assumed to be unitarily regular, which is a mild technical condition.
The central result in the first part is a strengthened version of the equivariant McDuff-type theorem, where equivariant tensorial absorption can be achieved with respect to so-called very strong cocycle conjugacy.

Secondly, we establish completely new results within the theory. 
This mainly concerns how equivariantly $\CZ$-stable absorption can be reduced to equivariantly UHF-stable absorption with respect to a given semi-strongly self-absorbing action.
Combining these abstract results with known uniqueness theorems due to Matui and Izumi-Matui, we obtain the following main result.
If $G$ is a torsion-free abelian group and $\CD$ is one of the known strongly self-absorbing \cstar-algebras, then strongly outer $G$-actions on $\CD$ are unique up to (very strong) cocycle conjugacy.
This is new even for $\IZ^3$-actions on the Jiang-Su algebra.
\end{abstract}

\maketitle

\tableofcontents


\section*{Introduction}

\noindent
This is a further continuation of my previous papers \cite{Szabo16ssa, Szabo16ssa2}, which introduced and studied (semi-)\-strongly self-absorbing \cstar-dynamical systems.
The motivation for studying such objects comes from the fundamental importance of strongly self-absorbing \cstar-algebras \cite{TomsWinter07} in the Elliott program.
For a more detailed description of this motivation and some history of the classification of group actions on \cstar-algebras and $\mathrm{W}^*$-algebras, the reader is referred to the introductions of the previous papers \cite{Szabo16ssa, Szabo16ssa2} and the references therein. 
A survey article \cite{Izumi10} by Izumi on these topics is especially noteworthy for anyone interested in the classification problem for group actions on operator algebras.

The first \cite{Szabo16ssa} of the previous papers provided an equivariant McDuff-type theorem characterizing equivariant tensorial absorption of (semi-)strongly self-absorb\-ing actions, generalizing classical results of R{\o}rdam \cite[Chapter 7, Section 2]{Rordam}, Toms-Winter \cite{TomsWinter07} and Kirchberg \cite{Kirchberg04}.
The second paper \cite{Szabo16ssa2} generalized some other classical results about strongly self-absorbing \cstar-algebras to the equivariant context, such as a stronger uniqueness theorem for certain equivariant $*$-homomorphisms by Dadarlat-Winter \cite{DadarlatWinter09} and permanence properties for the class of \cstar-algebras absorbing a fixed strongly self-absorbing \cstar-algebra.
In \cite{Szabo16ssa2}, the more sophisticated results could only be proved for semi-strongly self-absorbing actions that are unitarily regular. Simply put, this is an equivariant analog of the \cstar-algebraic property that the unitary commutator subgroup is in the connected component of the unit.
For a semi-strongly self-absorbing $G$-action $\gamma$, unitary regularity has been shown to be equivalent to the statement that the separable, $\gamma$-absorbing $G$-\cstar-dynamical systems are closed under equivariant extensions; see \cite[Section 4]{TomsWinter07} and \cite[Section 4]{Kirchberg04} for the corresponding classical results.
At present, it is open whether semi-strongly self-absorbing actions are automatically unitarily regular. However, $\gamma$ is unitarily regular if it is equivariantly $\CZ$-stable, which in turn is often automatic for discrete amenable acting groups, but not in general. 
In particular, the equivariant analog of the main result of \cite{Winter11} is not true in general; see \cite[5.4]{Szabo16ssa2}.
Apart from considering equivariant extensions, it is a theme throughout \cite{Szabo16ssa2} that for unitarily regular and semi-strongly self-absorbing actions, statements involving certain approximations by sequences can be smoothed out and strengthened to approximations by continuous paths. 

This is pursued further within the first half of this paper, where we improve the equivariant McDuff-type theorem from \cite{Szabo16ssa} in the unitarily regular case. 
Namely, we show that for a unitarily regular and semi-strongly self-absorbing action $\gamma: G\curvearrowright\CD$ and another action $\alpha: G\curvearrowright A$ on a separable \cstar-algebra, the McDuff condition implies that $\alpha$ and $\alpha\otimes\gamma$ are very strongly cocycle conjugate.
This means that they are conjugate modulo a cocycle that can be approximated by a continuous path of coboundaries starting at the unit; see Definition \ref{defi:conjugacies}\ref{item:vscc} and Theorem \ref{thm:strong-absorption}.
In the case of compact acting groups, one moreover gets that $\alpha$ and $\alpha\otimes\gamma$ are in fact conjugate; see Theorem \ref{thm:McDuff-compact}.

In the second half of the paper, some new results are obtained within the theory of semi-strongly self-absorbing actions. 
In section 4, we prove that for an action, the property of being semi-strongly self-absorbing can be detected by considering the restrictions with respect to an exhausting sequence of open subgroups of the acting group.
The same holds for the property of tensorially absorbing a given semi-strongly self-absorbing action; see Theorem \ref{thm:reduction-subgroups}.
This arises as a fairly straightforward consequence of the characterizations of these properties as approximate ones in previous work, combined with reindexation arguments.
In section 5, we prove that under the assumption of equivariant $\CZ$-stability, such properties can furthermore be detected after stabilizing with the trivial actions on UHF algebras of infinite type; see Theorem \ref{thm:reduction-Z}. 
In particular, this provides a way to reduce the classification of certain group actions on strongly self-absorbing \cstar-algebras to the classification of their UHF-stabilizations.
This is somewhat reminiscent of the main thrust of the methods developed by Winter in \cite{Winter14Lin}, which gave great impulse to the Elliott program, albeit the techniques developed in this paper have a much more narrow range of applicability in comparison. 

In section 6, these abstract results are then applied in combination with some known classification results to obtain the following uniqueness theorem for pointwise strongly outer actions on strongly self-absorbing \cstar-algebras. This constitutes the main result of the paper.

\begin{theoreme}
Let $\CD$ be a strongly self-absorbing \cstar-algebra satisfying the UCT. Let $G$ be a countable, torsion-free abelian group. Then any two pointwise strongly outer $G$-actions on $\CD$ are very strongly cocycle conjugate.
\end{theoreme}

We remark that, on a conceptual level, this type of result resembles Ocneanu's uniqueness theorem (see \cite{Ocneanu85}) for outer actions of amenable groups on the hyperfinite II${}_1$-factor. So in a sense, if one regards a strongly self-absorbing as a close \cstar-algebra analog of the hyperfinite II${}_1$-factor, one might call the above an Ocneanu-type uniqueness theorem.

Results of Matui \cite{Matui08, Matui11} and Izumi-Matui \cite{IzumiMatui10} have previously shown that the above is true for $\IZ^d$-actions on all the known strongly self-absorbing \cstar-algebras except for the Jiang-Su algebra $\CZ$. Sato \cite{Sato10} has shown such a uniqueness for $\IZ$-actions on $\CZ$, and Matui-Sato \cite{MatuiSato12_2} have extended this also to $\IZ^2$-actions on $\CZ$. 
We note that the uniqueness for actions of the Klein bottle group $\IZ\rtimes_{-1}\IZ$ is also known by further work of Matui-Sato \cite{MatuiSato14} on UHF algebras as well as $\CZ$; this was the first classification result for actions of non-abelian infinite groups on stably finite \cstar-algebras. 
Curiously, the known methods for showing a uniqueness result as above get increasingly difficult to implement with increasingly complicated acting groups, and even the uniqueness for pointwise strongly outer $\IZ^3$-actions on $\CZ$ has previously been open. 

Our main result essentially follows from three main ingredients:\ 
firstly, the known uniqueness for $\IZ^d$-actions on UHF-stable strongly self-absorbing \cstar-algebras mentioned above, which forms the basis of our argument;
secondly, the reduction theorems proved in sections 4 and 5 based on the abstract theory of semi-strongly self-absorbing actions;
and thirdly a result of Matui-Sato \cite[4.11]{MatuiSato14} asserting that pointwise strongly outer actions like above are automatically equivariantly $\CZ$-stable. See also a more recent paper \cite{Sato16} of Sato for a much more general $\CZ$-stability result.

It seems natural to expect that the known uniqueness results for $\IZ^d$-actions from \cite{Matui08, Matui11, IzumiMatui10} could be reproved abstractly within the common framework of semi-strongly self-absorbing actions, and without requiring the UCT assumption. 
It also seems plausible that this should in fact be possible for not necessarily abelian acting groups. For example, a uniqueness for actions like above seems feasible for (local) poly-$\IZ$ groups, considering an unpublished result of Izumi-Matui \cite{Izumi12OWR}. 
Considering moreover the $KK$-theoretically rigid situation for torsion-free amenable group actions on strongly self-absorbing \cstar-algebras showcased in \cite[4.12 and 4.17]{Szabo16kp}, I would go as far as to conjecture the following Ocneanu-type uniqueness, which shall be pursued in subsequent work:

\begin{conjecturee}
Let $\CD$ be a strongly self-absorbing \cstar-algebra. Let $G$ be a countable, torsion-free amenable group. Then any two pointwise strongly outer $G$-actions on $\CD$ are very strongly cocycle conjugate.
\end{conjecturee}

Note that a uniqueness theorem like this usually fails already for finite groups; see \cite{Izumi04, Izumi04II} for range results of outer cyclic group actions on $\CO_2$. 
Since the computation of the equivariant $KK$-theory of an action via the Baum-Connes assembly map requires one to consider all the finite subgroups of the acting group, it is natural to expect that the above conjecture should fail outside the torsion-free case.
Concerning non-amenable groups, actions are known not to be rigid.
On the one hand, for a given non-amenable group $G$, an argument in a paper of Jones \cite{Jones83} implies that for any finite strongly self-absorbing \cstar-algebra $\CD$, the noncommutative Bernoulli-shift on $\bigotimes_G\CD$ does not absorb the trivial $G$-action on $\CD$; this yields two pointwise strongly outer $G$-actions on $\CD$ that are not cocycle conjugate.
On the other hand, a recent paper of Gardella-Lupini \cite{GardellaLupini17} shows that rigidity fails much more spectacularly upon assuming that $G$ has property (T).

\bigskip
{\bf Acknowledgement.} The work presented in this paper has benefited from a visit to the Department of Mathematics at the University of Kyoto in January 2016, and I would like to express my gratitude to Masaki Izumi for the hospitality and support.


\section{Preliminaries}

\begin{nota}
Unless specified otherwise, we will stick to the following notational conventions in this paper:
\begin{itemize}
\item The symbol $\alpha$ is used for a continuous action $\alpha: G\curvearrowright A$ of a locally compact group $G$ on a \cstar-algebra $A$. By slight abuse of notation, we will also write $\alpha: G\to\Aut(\CM(A))$ for the unique strictly continuous extension.
\item For an action $\alpha: G\curvearrowright A$, $A^\alpha$ denotes the fixed-point algebra of $A$.
\item If $(X,d)$ is some metric space with elements $a,b\in X$, then we write $a=_\eps b$ as a shortcut for $d(a,b)\leq\eps$.
\item By a unitary path in a \cstar-algebra $A$ we shall understand a norm-continuous map from $[0,1]$ to $\CU(\tilde{A})$.
\end{itemize}
\end{nota}

First we recall the notion of $1$-cocycles for actions on \cstar-algebras and their cocycle perturbations. Note that we are adding a new refinement in this paper, given by so-called asymptotic coboundaries, very strong exterior equivalence and very strong cocycle conjugacy.

\begin{defi}[see {\cite[3.2]{PackerRaeburn89} and \cite[1.3, 1.6]{Szabo16ssa}}]
\label{def:scc}
Let $\alpha: G\curvearrowright A$ be an action. Consider a strictly continuous map $w: G\to\CU(\CM(A))$.
\begin{enumerate}[label=(\roman*),leftmargin=*] 
\item $w$ is called an $\alpha$-1-cocycle (or just $\alpha$-cocycle), if one has $w_g\alpha_g(w_h)=w_{gh}$ for all $g,h\in G$.
In this case, the map $\alpha^w: G\to\Aut(A)$ given by $\alpha_g^w=\ad(w_g)\circ\alpha_g$ is again an action, and is called a cocycle perturbation of $\alpha$. Two $G$-actions on $A$ are called exterior equivalent if one of them is a cocycle perturbation of the other.
\item Assume that $w$ is an $\alpha$-1-cocycle. It is called an approximate coboundary, if there exists a sequence of unitaries $x_n\in\CU(\CM(A))$ such that $x_n\alpha_g(x_n^*) \stackrel{n\to\infty}{\longrightarrow} w_g$ in the strict topology for all $g\in G$ and uniformly on compact subsets of $G$. Two $G$-actions on $A$ are called strongly exterior equivalent, if one of them is a cocycle perturbation of the other via an approximate coboundary.
\item Assume that $w$ is an $\alpha$-1-cocycle. It is called an asymptotic coboundary, if there exists a strictly continuous path of unitaries $x: [0,\infty)\to\CU(\CM(A))$ with $x_0=\eins$ such that $x_t\alpha_g(x_t^*) \stackrel{t\to\infty}{\longrightarrow} w_g$ in the strict topology for all $g\in G$ and uniformly on compact subsets of $G$. Two $G$-actions on $A$ are called very strongly exterior equivalent, if one of them is a cocycle perturbation of the other via an asymptotic coboundary.
\end{enumerate}
\end{defi}

Analogously, let us consider the generalization of these equivalence relations to cocycle actions:

\begin{defi}[see {\cite[3.1]{PackerRaeburn89}} for (i)]
Let $(\alpha,u), (\beta,w): G\curvearrowright A$ be two cocycle actions. 
\begin{enumerate}[label=(\roman*),leftmargin=*] 
\item The pairs $(\alpha,u)$ and $(\beta,w)$ are called exterior equivalent, if there is a strictly continuous map $v: G\to\CU(\CM(A))$ satisfying $\beta_g=\ad(v_g)\circ\alpha_g$ and $w(s,t)=v_s\alpha_s(v_t)u(s,t)v_{st}^*$ for all $g,s,t\in G$. \label{item:ee}
\item The pairs $(\alpha,u)$ and $(\beta,w)$ are called strongly exterior equivalent, if there is a map $v: G\to\CU(\CM(A))$ as in \ref{item:ee} such that there is a sequence of unitaries $x_n\in\CU(\CM(A))$ with $x_n\alpha_g(x_n^*) \stackrel{n\to\infty}{\longrightarrow} v_g$ in the strict topology for all $g\in G$ and uniformly on compact subsets of $G$. \label{item:see}
\item The pairs $(\alpha,u)$ and $(\beta,w)$ are called very strongly exterior equivalent, if there is a map $v: G\to\CU(\CM(A))$ as in \ref{item:ee} such that there is a strictly continuous path of unitaries $x: [0,\infty)\to\CU(\CM(A))$ with $x_0=\eins$ and $x_t\alpha_g(x_t^*) \stackrel{t\to\infty}{\longrightarrow} v_g$ in the strict topology for all $g\in G$ and uniformly on compact subsets of $G$. \label{item:vsee}
\end{enumerate}
\end{defi}

We recall several notions that describe how one can identify two cocycle actions on \cstar-algebras. We note that condition \ref{item:vscc} below is a new definition and a natural strengthening of the notion of strong cocycle conjugacy originally introduced by Izumi-Matui in \cite{IzumiMatui10}. 

\begin{defi} \label{defi:conjugacies}
Two cocycle actions $(\alpha,u): G\curvearrowright A$ and $(\beta,w): G\curvearrowright B$ are called 
\begin{enumerate}[label=(\roman*),leftmargin=*]
\item conjugate, if there is an equivariant isomorphism $\phi: (A,\alpha,u)\to (B,\beta,w)$. In this case, we write $(\alpha,u)\cong (\beta,w)$. \label{item:c}
\item cocycle conjugate, if there is an isomorphism $\phi: A\to B$ such that $(\phi\circ\alpha\circ\phi^{-1}, \phi\circ u)$ is exterior equivalent to $(\beta,w)$. In this case, we write $(\alpha,u)\cc (\beta,w)$. \label{item:cc}
\item strongly cocycle conjugate, if there is an isomorphism $\phi: A\to B$ such that $(\phi\circ\alpha\circ\phi^{-1}, \phi\circ u)$ is strongly exterior equivalent to $(\beta,w)$. In this case, we write $(\alpha,u)\scc (\beta,w)$. \label{item:scc}
\item very strongly cocycle conjugate, if there is an isomorphism $\phi: A\to B$ such that $(\phi\circ\alpha\circ\phi^{-1}, \phi\circ u)$ is very strongly exterior equivalent to $(\beta,w)$. In this case, we write $(\alpha,u)\vscc (\beta,w)$. \label{item:vscc}
\end{enumerate}
If $\alpha$ and $\beta$ are genuine actions, then we omit the 2-cocycles from this notation.
\end{defi}

\begin{defi}[see {\cite[1.1]{Kirchberg04} and \cite[1.7, 1.9, 1.10]{Szabo16ssa}}] 
Let $A$ be a \cstar-algebra and $(\alpha,u): G\curvearrowright A$ a cocycle action of a locally compact group $G$. 
\begin{enumerate}[label={\textup{(\roman*)}},leftmargin=*]
\item The sequence algebra of $A$ is given as 
\[
A_\infty = \ell^\infty(\IN,A)/\set{ (x_n)_n \mid \lim_{n\to\infty}\| x_n\|=0}.
\]
There is a standard embedding of $A$ into $A_\infty$ by sending an element to its constant sequence. We shall always identify $A\subset A_\infty$ this way, unless specified otherwise.
\item Suppose $u=\eins$. Pointwise application of $\alpha$ on representing sequences defines a (not necessarily continuous) $G$-action $\alpha_\infty$ on $A_\infty$. Let
\[
A_{\infty,\alpha} = \set{ x\in A_\infty \mid [g\mapsto\alpha_{\infty,g}(x)]~\text{is continuous} }
\]
be the continuous part of $A_\infty$ with respect to $\alpha$.
\item For some \cstar-subalgebra $B\subset A_\infty$, the (corrected) relative central sequence algebra is defined as
\[
F(B,A_\infty) = (A_\infty\cap B')/\ann(B,A_\infty).
\]
\item Suppose that $B\subset A_\infty$ is an $\alpha_\infty$-invariant \cstar-subalgebra closed under multiplication with the unitaries $\set{u(g,h)}_{g,h\in G}$. Then the map $\alpha_\infty: G\to\Aut(A_\infty)$ given by componentwise application of $\alpha$ induces a (not necessarily continuous) $G$-action $\tilde{\alpha}_\infty$ on $F(B,A_\infty)$. Let
\[
F_\alpha(B,A_\infty) = \set{ y\in F(B,A_\infty) \mid [g\mapsto\tilde{\alpha}_{\infty,g}(y)]~\text{is continuous} }
\]
be the continuous part of $F(B,A_\infty)$ with respect to $\alpha$.
\item In case $B=A$, we write $F(A,A_\infty)=F_\infty(A)$ and $F_\alpha(A,A_\infty)=F_{\infty,\alpha}(A)$.
\end{enumerate}
\end{defi}

\begin{nota}[\see{Szabo16ssa2}{1.14}] \label{1-unitaries}
Let $G$ be a second-countable, locally compact group. Let $A$ be a \cstar-algebra and $\alpha: G\curvearrowright A$ an action. For $\eps>0$ and a compact set $K\subset G$, define the closed set
\[
A^\alpha_{\eps,K}=\set{a\in A \mid \|\alpha_g(a)-a\|\leq\eps~\text{for all}~g\in K} \subset A.
\]
If $A$ is unital, then also consider
\[
\CU(A^\alpha_{\eps,K}) = \CU(A)\cap A^\alpha_{\eps,K}
\]
and
\[
\CU_0(A^\alpha_{\eps,K}) = \set{ u(1)\in\CU(A^\alpha_{\eps,K}) \mid u: [0,1]\to\CU(A^\alpha_{\eps,K})~\text{continuous,}~u(0)=\eins}.
\]
\end{nota}

\begin{defi}[\cf{Szabo16ssa2}{2.1}] \label{Gue}
Let $G$ be a second-countable, locally compact group, $A$ and $B$ two \cstar-algebras and $\alpha: G\curvearrowright A$ and $\beta: G\curvearrowright B$ two actions.  Let $\phi_1, \phi_2: (A,\alpha)\to (B,\beta)$ be two  equivariant $*$-homomorphisms. 
\begin{enumerate}[label={\textup{(\roman*)}},leftmargin=*]
\item We say that $\phi_1$ and $\phi_2$ are approximately $G$-unitarily equivalent, if for every $\eps>0$, every finite set $F\fin A$ and compact set $K\subset G$, there is a unitary $v\in \CU\big(\tilde{B}^\beta_{\eps,K}\big)$ such that $\|\phi_2(x)-v\phi_1(x)v^*\|\leq\eps$ for all $x\in F$. We write $\phi_1\ue{G}\phi_2$. \label{Gue:1}
\item Assume that $A$ is separable. We say that $\phi_1$ and $\phi_2$ are strongly asymptotically $G$-unitarily equivalent, if for every $\eps_0>0$ and compact set $K_0\subset G$, there is a continuous path of unitaries $w: [1,\infty)\to\CU\big(\tilde{B}^\beta_{\eps_0,K_0}\big)$ satisfying $w(1)=\eins_B$,
\[
\phi_2(x)=\lim_{t\to\infty} w(t)\phi_1(x)w(t)^*\quad\text{for all}~x\in A,
\]
and
\[
\lim_{t\to\infty}~\max_{g\in K}~\|\beta_g(w(t))-w(t)\|=0 \quad\text{for every compact set}~K\subset G.
\]
\label{Gue:2}
\item Suppose that $G$ is compact. Then by averaging, we may assume in \ref{Gue:1} that $v$ is in the fixed point algebra $\tilde{B}^\beta$. We may also assume in \ref{Gue:2} that the path $w$ takes values in the fixed point algebra $\tilde{B}^\beta$.
\item If we disregard equivariance and consider $G=\set{1}$ in \ref{Gue:1} and \ref{Gue:2}, then we say that the maps $\phi_1$ and $\phi_2$ are approximately unitarily equivalent or strongly asymptotically unitarily equivalent, respectively.
\end{enumerate}
\end{defi}

\begin{defi}[\cf{Szabo16ssa}{3.1, 4.1}]
Let $\CD$ be a separable, unital \cstar-algebra and $G$ a second-countable, locally compact group. Let $\gamma: G\curvearrowright\CD$ be an action. We say that $\gamma$ is
\begin{enumerate}[label={(\roman*)},leftmargin=*]
\item strongly self-absorbing, if the equivariant first-factor embedding
\[
\id_\CD\otimes\eins_\CD: (\CD,\gamma)\to (\CD\otimes\CD,\gamma\otimes\gamma)
\]
is approximately $G$-unitarily equivalent to an isomorphism.
\item semi-strongly self-absorbing, if it is strongly cocycle conjugate to a strongly self-absorbing action.
\end{enumerate}
\end{defi}

Let us recall some results from \cite{Szabo16ssa}. 

\begin{theorem}[\see{Szabo16ssa}{4.6}] \label{thm:sssa}
Let $\CD$ be a separable, unital \cstar-algebra and $G$ a second-countable, locally compact group. Let $\gamma: G\curvearrowright\CD$ be an action. The following are equivalent:
\begin{enumerate}[label=\textup{(\roman*)},leftmargin=*] 
\item $\gamma$ is semi-strongly self-absorbing;
\item $\gamma$ has approximately $G$-inner half-flip and there exists a unital and equivariant $*$-homomorphism from $(\CD,\gamma)$ to $(\CD_{\infty,\gamma}\cap\CD',\gamma_\infty)$;
\item $\gamma$ has approximately $G$-inner half-flip and $\gamma\scc\gamma^{\otimes\infty}$.
\end{enumerate}
\end{theorem}

\begin{theorem}[see {\cite[3.7, 4.7]{Szabo16ssa}}]
\label{equi-McDuff}
Let $G$ be a second-countable, locally compact group. Let $A$ be a separable \cstar-algebra and $(\alpha,u): G\curvearrowright A$ a cocycle action. Let $\CD$ be a separable, unital \cstar-algebra and $\gamma: G\curvearrowright\CD$ a semi-strongly self-absorbing action. The following are equivalent:
\begin{enumerate}[label=\textup{(\roman*)},leftmargin=*] 
\item $(\alpha,u)\scc(\alpha\otimes\gamma,u\otimes\eins)$. \label{equi-McDuff1}
\item $(\alpha,u)\cc(\alpha\otimes\gamma,u\otimes\eins)$.  \label{equi-McDuff2}
\item There exists a unital and equivariant $*$-homomorphism from $(\CD,\gamma)$ to $\big( F_{\infty,\alpha}(A), \tilde{\alpha}_\infty \big)$. \label{equi-McDuff3}
\end{enumerate}
\end{theorem}

\begin{reme}
An action $\alpha$ satisfying condition \ref{equi-McDuff}\ref{equi-McDuff1} is called $\gamma$-absorbing. 
\end{reme}

We shall also recall the notion of a unitarily regular action.

\begin{defi}[\see{Szabo16ssa2}{1.17}] \label{unireg}
Let $G$ be a second-countable, locally compact group. Let $A$ be a unital \cstar-algebra and $\alpha: G\curvearrowright A$ an action. We say that $\alpha$ is unitarily regular, if for every compact set $K\subset G$ and $\eps>0$, there exists $\delta>0$ such that $uvu^*v^*\in\CU_0(A^\alpha_{\eps,K})$ for every $u,v\in\CU(A^\alpha_{\delta,K})$.
\end{defi}

\begin{theorem}[\see{Szabo16ssa}{2.2}] \label{thm:saGi-hf}
If a semi-strongly self-absorbing action $\gamma: G\curvearrowright\CD$ is unitarily regular, then $\gamma$ has strongly asymptotically $G$-inner half-flip. In particular, the half-flip can be approximately implemented by unitaries in $\CU_0\big( (\CD\otimes\CD)^{\gamma\otimes\gamma}_{\eps,K} \big)$ for arbitrarily small $\eps>0$ and large compact sets $K\subset G$.
\end{theorem}


\section{Strengthened McDuff-type theorem}

The following is a continuous generalization of a key technical Lemma from \cite{Szabo16ssa}. Its proof is fairly analogous and employs a few straightforward modifications.

\begin{lemma}[\cf{Szabo16ssa}{2.1}] 
\label{one-side}
Let $G$ be a second-countable, locally compact group. Let $(\alpha,u): G\curvearrowright A$ and $(\beta,w): G\curvearrowright B$ be two cocycle actions on separable \cstar-algebras. Let $\phi: (A,\alpha,u)\to (B,\beta,w)$ be an injective, non-degenerate and equivariant $*$-homomorphism. Assume the following:

For every $\eps>0$, compact subset $K\subset G$ and finite subsets $F_A\fin A, F_B\fin B$, there exists a unitary path $z: [0,1]\to\CU(\tilde{B})$ with $z_0=\eins$ satisfying
\begin{enumerate}[label=\textup{(\ref*{one-side}\alph*)}]
\item $\|[z_t, \phi(a)]\|\leq\eps$ for every $a\in F_A$ and $0\leq t\leq 1$. \label{eq:oneside-a}
\item $\dist(z_1^*bz_1, \phi(A))\leq\eps$ for every $b\in F_B$. \label{eq:oneside-b}
\item $\phi(a)\beta_g(z_t) =_\eps \phi(a)z_t$ for every $g\in K$, $a\in F_A$ and $0\leq t\leq 1$. \label{eq:oneside-c}
\end{enumerate}
Then $\phi$ is strongly asymptotically unitarily equivalent to an isomorphism $\psi: A\to B$ inducing very strong cocycle conjugacy between $(\alpha,u)$ and $(\beta,w)$.
\end{lemma}
\begin{proof}
We first comment that by the non-degeneracy of $\phi$, one can replace the elements $\phi(a)$ in \ref{eq:oneside-c}, for $a\in F_A$, by any element $b\in F_B$.

Let $\set{a_n}_{n\in\IN}\subset A$ and $\set{b_n}_{n\in\IN}\subset B$ be dense sequences. Since $G$ is $\sigma$-compact, write $G=\bigcup_{n\in\IN} K_n$ for an increasing union of compact subsets $1_G\in K_n$. We are going to add paths of unitaries to $\phi$ step by step:

In the first step, choose some $a_{1,1}\in A$ and $z^{(1)}: [0,1]\to\CU(\tilde{B})$ with $z^{(1)}_0=\eins$ such that for all $0\leq t\leq 1$, we have
\begin{itemize}
\item $z^{(1)*}_1b_1z^{(1)}_1 =_{1/2} \phi(a_{1,1})$;
\item $\|[z^{(1)}_t,\phi(a_1)]\|\leq 1/2$;
\item $b_1\beta_g(z^{(1)}_t)=_{1/2} b_1z^{(1)}_t$ for all $g\in K_1$.
\end{itemize}
In the second step, choose $a_{2,1},a_{2,2}\in A$ and $z^{(2)}: [0,1]\to\CU(\tilde{B})$ with $z^{(2)}_0=\eins$ such that for every $0\leq t\leq 1$ we have
\begin{itemize}
\item $z^{(2)*}_1(z^{(1)*}_1b_jz^{(1)}_1)z^{(2)}_1 =_{1/4} \phi(a_{2,j})$ for $j=1,2$;
\item $\|[z^{(2)}_t,\phi(a_j)]\|\leq 1/4$ for $j=1,2$;
\item $\|[z^{(2)}_t,\phi(a_{1,1})]\|\leq 1/4$;
\item $(b_jz^{(1)}_1)\beta_{g}(z^{(2)}_t) =_{1/4} (b_jz^{(2)}_1)z^{(2)}_t$ for all $g\in K_2$ and $j=1,2$.
\end{itemize}
Now assume that for some $n\in\IN$, we have found $z^{(1)},\dots,z^{(n)}: [0,1]\to\CU(\tilde{B})$ and $\set{a_{m,j}}_{m,j\leq n}\subset A$ satisfying for every $0\leq t\leq 1$ that
\begin{equation} \label{eq1:c1}
z^{(n)*}_1(z^{(n-1)*}_1\cdots z^{(1)*}_1b_jz^{(1)}_1\cdots z^{(n-1)}_1)z^{(n)}_1 =_{2^{-n}} \phi(a_{n,j})~ \text{for}~ j\leq n;
\end{equation}
\begin{equation} \label{eq1:c2}
\|[z^{(n)}_t,\phi(a_j)]\|\leq 2^{-n}~ \text{for}~ j\leq n;
\end{equation}
\begin{equation} \label{eq1:c3}
\|[z^{(n)}_t,\phi(a_{m,j})]\|\leq 2^{-n}~ \text{for}~ m<n~ \text{and}~ j<m ;
\end{equation}
\begin{equation} \label{eq1:c4}
(b_jz^{(1)}_1\cdots z^{(n-1)}_1)\beta_{g}(z^{(n)}_t) =_{2^{-n}} (b_jz^{(1)}_1\cdots z^{(n-1)}_1)z^{(n)}_t~ \text{for}~ g\in K_n~ \text{and}~ j\leq n.
\end{equation}
Then we can again apply our assumptions to find $z^{(n+1)}: [0,1]\to\CU(\tilde{B})$ with $z^{(n+1)}_0=\eins$ and $\set{a_{n+1,j}}_{j\leq n+1}\subset A$ so that for every $0\leq t\leq 1$ we have
\begin{itemize}
\item $z^{(n+1)*}_1(z^{(n)*}_1\cdots z^{(1)*}_1 b_j z^{(1)}_1\cdots z^{(n)}_1)z^{(n+1)}_1 =_{2^{-(n+1)}} \phi(a_{n+1,j})$ for $j\leq n+1$;
\item $\|[z^{(n+1)}_t,\phi(a_j)]\|\leq 2^{-(n+1)}$ for $j\leq n+1$;
\item $\|[z^{(n+1)}_t,\phi(a_{m,j})]\|\leq 2^{-(n+1)}$ for $m<n+1$ and $j<n+1$;
\item $(b_j z^{(1)}_1\cdots z^{(n)}_1)\beta_{g}(z^{(n+1)}_t) =_{2^{-(n+1)}} (b_j z^{(1)}_1\cdots z^{(n)}_1)z^{(n+1)}_t$ for all $g\in K_{n+1}$ and $j\leq n+1$.
\end{itemize}
Carry on inductively. We define a norm-continuous path of unitaries $x: [0,\infty)\to\CU(\tilde{B})$ via $x_t = z^{(1)}_1\cdots z^{(n)}_1 z^{(n+1)}_{t-n}$ for $n\geq 0$ with $n\leq t\leq n+1$.
Note that $x_0=\eins$, and this map is well-defined since every  path $z^{(n)}$ starts at the unit. Similarly we define a point-norm continuous path of  $*$-homomorphisms $\psi_t: A\to B$ for $t\geq 0$ via $\psi_t = \ad(x_t)\circ\phi$. 

Now let us observe a number of facts:\ 
By condition \eqref{eq1:c2}, the net $(\psi_t(a_j))_{t\geq 0}$ is Cauchy for all $j\in\IN$. Since the set $\set{a_j}_{j\in\IN}\subset A$ is dense, this implies that the net $(\psi_t)_{t\geq 0}$ converges to some $*$-homomorphism $\psi: A\to B$ in the point-norm topology. Then $\psi$ is clearly strongly asymptotically unitarily equivalent to $\phi$.

Exactly as in the proof of \cite[2.1]{Szabo16ssa}, one deduces from \eqref{eq1:c1} and \eqref{eq1:c3} that $\psi$ is surjective, and hence an isomorphism.

By condition \eqref{eq1:c4}, we have that the assignments $t\mapsto b_j\cdot x_t \beta_g(x_t^*)$ yield Cauchy nets for every $j\in\IN$ and $g\in G$, with uniformity on compact subsets of $G$.
Since $\set{b_j}_{j\in\IN}\subset B$ is dense, it follows that every path of functions of the form $[g\mapsto b\cdot x_t\beta_g(x_t^*)]$ (for $b\in B$) converges uniformly on compact sets. Since $\beta$ is point-norm continuous, it follows that the functions 
\[
g\mapsto \beta_g\Big( \beta_{g^{-1}}(b)^*\cdot x_t \beta_g(x_t^*) \Big)^* = x_t\beta_g(x_t^*)\cdot b
\]
must also converge uniformly on compact sets of $G$ as $t\to\infty$, for every $b\in B$.

It follows that the strict limit $v_g = \lim_{t\to\infty} x_t\beta_g(x_t^*) \in \CU(\CM(B))$ exists for every $g\in G$, and that this convergence is uniform on compact subsets of $G$. In particular, the assignment $g\mapsto v_g\in\CU(\CM(B))$ is strictly continuous. 

Exactly as in the proof of \cite[2.1]{Szabo16ssa}, it follows that $\psi\circ\alpha_g=\ad(v_g)\circ\beta_g\circ\psi$ and $v_g\beta_g(v_h)w(g,h)v_{gh}^*= \psi(u(g,h))$ for all $g,h\in G$. This finishes the proof.
\end{proof}

Here comes the main result of this section, which is an improved version of the equivariant McDuff theorem \cite[3.7, 4.7]{Szabo16ssa} for unitarily regular actions:

\begin{theorem} \label{thm:strong-absorption}
Let $G$ be a second-countable, locally compact group. Let $\gamma: G\curvearrowright\CD$ be a unitarily regular and semi-strongly self-absorbing action on a separable, unital \cstar-algebra. Let $(\alpha,u): G\curvearrowright A$ be a cocycle action on a separable \cstar-algebra. Suppose that there exists a unital and equivariant $*$-homomorphism from $(\CD,\gamma)$ to $\big( F_{\infty,\alpha}(A), \tilde{\alpha}_\infty \big)$. Then the equivariant second-factor embedding 
\[
\eins_\CD\otimes\id_A: (A,\alpha,u)\to (\CD\otimes A, \gamma\otimes\alpha, \eins_\CD\otimes u)
\]
is strongly asymptotically unitarily equivalent to an isomorphism that induces very strong cocycle conjugacy. In particular, we have $(\alpha,u)\vscc (\alpha\otimes\gamma,u\otimes\eins_\CD)$.
\end{theorem}
\begin{proof}
The proof is very similar to \cite{Szabo16ssa}, apart from a small modification.

Keeping in mind \cite[1.11]{Szabo16ssa}, we have a natural isomorphism
\[
F(\eins_\CD\otimes A,(\CD\otimes A)_\infty)\cong F(\eins_\CD\otimes A, \big( (\CD\otimes A)^\sim \big)_\infty ).
\]
Denote by $\pi: \big( (\CD\otimes A)^\sim \big)_\infty\cap (\eins_\CD\otimes A)' \to F(\eins_\CD\otimes A,(\CD\otimes A)_\infty)$ the canonical surjection. 
Note that by assumption, we have an equivariant, unital $*$-homomorphism from $(\CD,\gamma)$ to $\big( F_{\infty,\alpha}(A), \tilde{\alpha}_\infty \big)$. Consider the canonical inclusions 
\[
F_{\infty}(A) ,~\CD ~\subset~ F(\eins_\CD\otimes A,(\CD\otimes A)_\infty),
\]
which define commuting \cstar-subalgebras. Since these inclusions are natural, they are equivariant with respect to the induced actions of $\alpha$, $\gamma$ and $\gamma\otimes\alpha$.
By assumption, it follows that we have a unital and equivariant $*$-homomorphism
\[
\phi: (\CD\otimes\CD,\gamma\otimes\gamma)\to \big( F_{\gamma\otimes\alpha}(\eins_\CD\otimes A,(\CD\otimes A)_\infty) , (\gamma\otimes\alpha)^\sim_\infty \big)
\]
satisfying $\phi(d\otimes\eins_\CD)\cdot (\eins_\CD\otimes a)=d\otimes a$ and $\phi(\eins_\CD\otimes d)\cdot (\eins_\CD\otimes a)\in\eins_{\CD}\otimes A_\infty$ for all $a\in A$ and $d\in\CD$.

Now let $\eps>0$, $F_\CD\fin\CD$, $F_A\fin A$ and $K\subset G$ be a compact set. Without loss of generality, assume that $F_\CD$ and $F_A$ consist of contractions.
Since $\gamma$ is unitarily regular, we can apply \ref{thm:saGi-hf} and choose a unitary path $v: [0,1]\to\CU(\CD\otimes\CD)$ with $v_0=\eins$ and
\[
\max_{0\leq t\leq 1}~ \max_{g\in K}~ \|v_t-(\gamma\otimes\gamma)_g(v_t)\|\leq\eps \quad\text{and}\quad v_1^*(d\otimes\eins_\CD)v_1 =_\eps \eins_\CD\otimes d
\]
for all $d\in F_\CD$. 
The unitary path 
\[
u: [0,1] \to \CU\Big( F_{\gamma\otimes\alpha}(\eins_\CD\otimes A,(\CD\otimes A)_\infty) \Big), \quad u_t=\phi(v_t)
\]
then satisfies
\[
u_1^*(d\otimes a)u_1 = \phi(v_1(d\otimes\eins_\CD)v_1^*)\cdot (\eins_\CD\otimes a) =_\eps \phi(\eins_\CD\otimes d)\cdot (\eins_\CD\otimes a) \in \eins_\CD\otimes A_\infty
\] 
for all $a\in A$ with $\|a\|\leq 1$ and $d\in F_\CD$, and moreover
\begin{equation} \label{eq2:fixed}
\|u_t-(\gamma\otimes\alpha)^\sim_{\infty,g}(u_t)\| \leq \|v_t-(\gamma\otimes\gamma)_g(v_t)\| \leq\eps
\end{equation}
for all $g\in K$ and $0\leq t\leq 1$.

Applying the unitary lifting theorem \cite[5.1]{Blackadar15}, we can find paths of unitaries $z^{(n)}: [0,1]\to\CU\big((\CD\otimes A)^\sim\big)$ with $z^{(n)}_0=\eins$ and such that \[
z=[(z^{(n)})]: [0,1]\to \big( (\CD\otimes A)^\sim \big)_\infty\cap (\eins_\CD\otimes A)'
\]
satisfies $u_t=\pi(z_t)$ for all $0\leq t\leq 1$. 
Note that each $u_t$ is a continuous element with respect to $(\gamma\otimes\alpha)^\sim_\infty$, so it follows that 
\[
[g\mapsto (\eins_\CD\otimes a)\cdot (\gamma\otimes\alpha)_{\infty,g}(z_t)]
\] 
is a norm-continuous map on $G$ for every $a\in A$. Using \cite[2.2]{Szabo16ssa}, we thus see that also
\[
g\mapsto \Big( (\eins_\CD\otimes a)\cdot (\gamma\otimes\alpha)_{g}(z^{(n)}_t) \Big)_{n\in\IN} \in \ell^\infty(\IN,\CD\otimes A)
\]
is continuous. In particular, we obtain a uniformly continuous map
\[
[0,1]\times K \to \ell^\infty(\IN,\CD\otimes A),\quad (t,g) \mapsto \Big( (\eins_\CD\otimes a)\cdot (\gamma\otimes\alpha)_{g}(z^{(n)}_t) \Big)_{n\in\IN}.
\]
From \eqref{eq2:fixed} it thus follows that
\begin{equation} \label{eq:u-fixed}
\limsup_{n\to\infty}~ \max_{g\in K}~ \max_{0\leq t\leq 1}~ \big\| (\eins_\CD\otimes a)\cdot \big( (\gamma\otimes\alpha)_{g}(z^{(n)}_t)-z^{(n)}_t \big) \big\| \leq \eps
\end{equation}
for all $a\in A$ with $\|a\|\leq 1$.

Moreover, as $z$ is a lift for $u$ by choice, we have
\[
\dist(z_1^*(d\otimes a)z_1, \eins_\CD\otimes A_\infty)\leq\eps
\] 
for all $a\in A$ with $\|a\|\leq 1$ and $d\in F_\CD$.
  
It follows that there exists some $n$ with
\begin{itemize}
\item $\dst \max_{0\leq t\leq 1}~\|[z^{(n)}_t, \eins_\CD\otimes a]\|\leq\eps$ for all $a\in F_A$;
\item $\dist( z^{(n)*}_1(d\otimes a)z^{(n)}_1, \eins_\CD\otimes A)\leq 2\eps$ for all $d\in F_\CD$ and $a\in F_A$;
\item $\dst\max_{g\in K}~\max_{0\leq t\leq 1}~\|(\eins_\CD\otimes a)\cdot \big( z^{(n)}_t-(\gamma\otimes\alpha)_g(z^{(n)}_t) \big)\| \leq 2\eps$.
\end{itemize}
So we have met the conditions of \ref{eq:oneside-a}, \ref{eq:oneside-b} and \ref{eq:oneside-c} for the equivariant embedding $\eins_\CD\otimes\id_A: (A,\alpha,u)\to (\CD\otimes A, \gamma\otimes\alpha, \eins_\CD\otimes u)$. The claim follows. 
\end{proof}


\section{Optimal McDuff-type theorem for compact groups}

In this section, we turn to the case of compact group actions.
Our main observation here is that, upon a close inspection of the proofs of \ref{one-side} and \ref{thm:strong-absorption}, one can further improve the absorption to conjugacy for compact group actions.

\begin{rem}
Let $G$ be a compact group and $\gamma: G\curvearrowright\CD$ a strongly self-absorbing action. 
It was observed in \cite[4.10]{Szabo16ssa} that an action $\alpha: G\curvearrowright A$ on a separable, unital \cstar-algebra is $\gamma$-absorbing if and only if $\alpha$ is conjugate to $\alpha\otimes\gamma$. This relies on the more general observation that on unital \cstar-algebras, two compact group actions are conjugate if and only if they are strongly cocycle conjugate. 
This, in turn, relied on an observation \cite[2.4]{Izumi04} of Izumi stating that cocycles close to the unit are coboundaries. It is not known whether this can be generalized to the non-unital case. 

For unitarily regular actions $\gamma$, however, it turns out that we can always obtain conjugacy between $\alpha$ and $\alpha\otimes\gamma$ upon a closer inspection of the proofs in the previous section.
The crucial part about unitary regularity here is that, within the proof of the equivariant McDuff theorem, one needs some kind of gadget to lift a $G$-invariant unitary to a $G$-invariant unitary under a certain quotient map; it also gives us the strong asymptotic $G$-unitary equivalence in the statement of \ref{thm:McDuff-compact}.
It would be natural to expect that conjugacy can be arranged without assuming unitary regularity, however this would for example presuppose a new way of proving the non-equivariant, non-unital McDuff theorem not relying on unitary regularity, which to the best of my knowledge (and despite some effort) does not yet exist.
\end{rem}

\begin{lemma} 
\label{one-side-compact}
Let $G$ be a second-countable, compact group. Let $\alpha: G\curvearrowright A$ and $\beta: G\curvearrowright B$ be two actions on separable \cstar-algebras. Let $\phi: (A,\alpha)\to (B,\beta)$ be an injective, non-degenerate and equivariant $*$-homomorphism. Suppose that for every $\eps>0$ and every pair of finite subsets $F_A\fin A, F_B\fin B$, there exists a unitary path $z: [0,1]\to\CU(\tilde{B}^\beta)$ with $z_0=\eins$ satisfying
\[
\|[z_t, \phi(a)]\|\leq\eps \quad\text{for every } a\in F_A \text{ and } 0\leq t\leq 1,
\]
and 
\[
\dist(z_1^*bz_1, \phi(A))\leq\eps \quad\text{for every } b\in F_B.
\]
Then $\phi$ is strongly asymptotically $G$-unitarily equivalent to an isomorphism. In particular, $\alpha$ and $\beta$ are conjugate.
\end{lemma}

\begin{proof}
Proving this is completely identical to the classical one-sided intertwining result \cite[2.3.5]{Rordam}. Proceed as in the proof of \ref{one-side}; since the paths $z^{(n)}$ take values in the fixed point algebra, one may simply omit everything related to the cocycles.
\end{proof}

\begin{theorem} \label{thm:McDuff-compact}
Let $G$ be a second-countable, compact group. Let $\gamma: G\curvearrowright\CD$ be a unitarily regular and strongly self-absorbing action on a separable, unital \cstar-algebra. Let $\alpha: G\curvearrowright A$ be an action on a separable \cstar-algebra. Suppose that there exists a unital and equivariant $*$-homomorphism from $(\CD,\gamma)$ to $\big( F_{\infty,\alpha}(A), \tilde{\alpha}_\infty \big)$. Then the equivariant second-factor embedding 
\[
\eins_\CD\otimes\id_A: (A,\alpha)\to (\CD\otimes A, \gamma\otimes\alpha)
\]
is strongly asymptotically $G$-unitarily equivalent to an isomorphism. In particular, $\alpha$ is conjugate to $\alpha\otimes\gamma$.
\end{theorem}
\begin{proof}
Proceed exactly as in the proof of \ref{thm:strong-absorption} until choosing the $*$-homomor\-phism $\phi$. We make the additional observation that by \cite[3.7]{Szabo16ssa2}, the canonical projection
\[
\pi: \big( (\CD\otimes A)^\sim \big)_\infty\cap (\eins_\CD\otimes A)' \to F(\eins_\CD\otimes A,(\CD\otimes A)_\infty)
\]
restricts to an equivariant, surjective $*$-homomorphism on the continuous parts
\[
\pi: \big( (\CD\otimes A)^\sim \big)_{\infty,\gamma\otimes\alpha}\cap (\eins_\CD\otimes A)' \to F_{\gamma\otimes\alpha}(\eins_\CD\otimes A,(\CD\otimes A)_\infty).
\]
As $G$ is compact, this implies that $\pi$ also becomes surjective after restricting it to the fixed-point algebras (see \cite[3.9]{BarlakSzabo16ss})
\begin{equation} \label{pi-restrict}
\pi: \big( (\CD\otimes A)^{\gamma\otimes\alpha,\sim} \big)_\infty\cap (\eins_\CD\otimes A)' \to F(\eins_\CD\otimes A,(\CD\otimes A)_\infty)^{(\gamma\otimes\alpha)^\sim}
\end{equation}
Now let $\eps>0$, $F_\CD\fin\CD$ and $F_A\fin A$ be given.
As $\gamma$ is strongly self-absorbing and unitarily regular, it has strongly asymptotically $G$-inner half-flip.
As $G$ is additionally compact, we can find a path of unitaries $v: [0,1]\to\CU\big( (\CD\otimes\CD)^{\gamma\otimes\gamma} \big)$ with $v_0=\eins$ and $v_1(d\otimes\eins_\CD)v_1^*=_\eps\eins_\CD\otimes d$ for all $d\in F_\CD$.
Then $u_t=\phi(v_t)$ yields a unitary path $u: [0,1]\to \CU\big( F(\eins_\CD\otimes A,(\CD\otimes A)_\infty)^{(\gamma\otimes\alpha)^\sim} \big)$ such that $u_1^*(d\otimes a)u_1$ has distance at most $\eps$ from $\eins_\CD\otimes A_\infty$ for all $d\in F_\CD$ and all $a\in A$ with $\|a\|\leq 1$.

Then this path lifts under the above restriction \eqref{pi-restrict} of $\pi$ by the unitary lifting theorem \cite[5.1]{Blackadar15}. That is, there exists a sequence of unitary paths $z^{(n)}: [0,1] \to \CU\big( (\CD\otimes A)^{\gamma\otimes\alpha,\sim} \big)$ representing $u$. But then for sufficiently large $n$, we necessarily have
\[
\max_{0\leq t\leq 1}~\|[z^{(n)}_t, \eins_\CD\otimes a]\|\leq\eps
\]
and
\[
\dist( z^{(n)*}_1(d\otimes a)z^{(n)}_1, \eins_\CD\otimes A)\leq 2\eps
\]
for all $d\in F_\CD$ and $a\in F_A$.

As $\eps, F_\CD, F_A$ were arbitrary, the assertion follows from \ref{one-side-compact}.
\end{proof}


\section{Reduction to subgroups}

In this section, we will study certain behavior of group actions in the case where the acting group arises as a union of open subgroups.

\begin{nota}
Let $G$ be a topological group with a distinguished subgroup $H\subset G$. Let $A$ be a \cstar-algebra and $(\alpha,u): G\curvearrowright A$ a cocycle action. Then we write $(\alpha,u)|_H: H\curvearrowright A$ for the cocycle $H$-action on $A$ that arises by restriction.
Let $B$ be another \cstar-algebra with a cocycle action $(\beta,w): G\curvearrowright B$, and let $\phi: (A,\alpha,u)\to (B,\beta,w)$ be a nondegenerate, equivariant $*$-homomorphism. Then we write $\phi|_H: (A,\alpha,u)|_H \to (B,\beta,w)|_H$ for the equivariant $*$-homomorphism between the restricted dynamical systems. (This is equal to $\phi$ as a map, but is viewed as an arrow in a different category.) For genuine actions, we will omit the $2$-cocycles in this notation.
\end{nota}

\begin{lemma} \label{subgroup:uniqueness}
Let $G$ be a second-countable, locally compact group, $A$ and $B$ two \cstar-algebras and $\alpha: G\curvearrowright A$ and $\beta: G\curvearrowright B$ two actions. Let $\phi_1, \phi_2: (A,\alpha)\to (B,\beta)$ be two  equivariant $*$-homomorphisms. Let $\set{G_n}_{n\in\IN}$ be an increasing family of open subgroups of $G$ such that $G=\bigcup_{n\in\IN} G_n$. Suppose that
\[
\phi_1|_{G_n} \ue{G_n} \phi_2|_{G_n} \quad\text{for all } n\in\IN.
\]
Then $\phi_1\ue{G}\phi_2$.
\end{lemma}
\begin{proof}
Let $\eps>0$, $F\fin A$ a finite subset and $K\subset G$ a compact subset. Since the increasing subgroups $G_n$ are open, it follows by compactness that there exists some $N\in\IN$ with $K\subset G_N$. As $\phi_1|_{G_N} \ue{G_N} \phi_2|_{G_N}$ by assumption, we find some
\[
v ~\in~ \CU\big( \tilde{B}^{\beta|_{G_N}}_{\eps,K} \big) ~=~ \CU\big( \tilde{B}^{\beta}_{\eps,K} \big)
\]
with $\|\phi_2(x)-v\phi_1(x)v^*\|\leq\eps$ for all $x\in F$. This finishes the proof.
\end{proof}

\begin{cor} \label{cor:half-flip-subgroup}
Let $G$ be a second-countable, locally compact group. Let $\CD$ be a separable, unital \cstar-algebra and $\gamma: G\curvearrowright\CD$ an action. Let $\set{G_n}_{n\in\IN}$ be an increasing family of open subgroups of $G$ such that $G=\bigcup_{n\in\IN} G_n$. Then $\gamma$ has approximately $G$-inner half-flip if and only if $\gamma|_{G_n}$ has approximately $G$-inner half-flip for every $n\in\IN$.
\end{cor}

\begin{rem} \label{rmk:cont-reps}
Let $G$ be a second-countable, locally compact group. Let $A$ be a separable \cstar-algebra and $\alpha: G\curvearrowright A$ an action. By the results in \cite[Section 3]{Szabo16ssa2}, we have a natural identification
\[
F_{\infty,\alpha}(A) = (A_{\infty,\alpha}\cap A')/\ann(A,A_{\infty,\alpha}).
\]
Moreover, the ideal $\ann(A,A_{\infty,\alpha})$ is a $G$-$\sigma$-ideal in $A_{\infty,\alpha}\cap A'$, which implies that the quotient map from $A_{\infty,\alpha}\cap A'$ to $F_{\infty,\alpha}(A)$ is strongly locally semi-split; see \cite[3.5, 3.6]{Szabo16ssa2}. 
Denote by $\ell^\infty_\alpha(\IN,A)\subset\ell^\infty(\IN,A)$ the \cstar-algebra containing those bounded sequences $(x_n)_n$ for which the map $[g\mapsto (\alpha_g(x_n))_n ]$ is norm-continuous. By a result of Brown \cite[2.1]{Brown00}, we have
\[
A_{\infty,\alpha} = \ell_\alpha^\infty(\IN,A)/c_0(\IN,A).
\]
\end{rem}

The following is nothing more than a fairly routine reindexation argument.

\begin{lemma} \label{lemma:F(A)-subgroup}
Let $G$ be a second-countable, locally compact group. Let $A$ be a separable \cstar-algebra with a cocycle action $(\alpha,u): G\curvearrowright A$. Let $\CD$ be a separable, unital \cstar-algebra and $\gamma: G\curvearrowright\CD$ an action.
Let $\set{G_n}_{n\in\IN}$ be an increasing family of open subgroups of $G$ such that $G=\bigcup_{n\in\IN} G_n$. Suppose that for every $n\in\IN$, there exists a unital and equivariant $*$-homomorphism from $(\CD,\gamma|_{G_n})$ to $\big(F_{\infty,\alpha|_{G_n}}(A),\tilde{\alpha}_\infty|_{G_n}\big)$.
Then there exists a unital and equivariant $*$-homomorphism from $(\CD,\gamma)$ to $\big(F_{\infty,\alpha}(A),\tilde{\alpha}_\infty\big)$.
\end{lemma}
\begin{proof}
First let us observe that it suffices to consider only the case of genuine actions. Consider the Hilbert space $\CH=\ell^2(\IN)\bar{\otimes} L^2(G)$ and let $\delta: G\curvearrowright\CK(\CH)$ be the unitarily implemented action that is induced by the left-regular representation of $G$ on $L^2(G)$.
By \cite[1.10]{Szabo16ssa} and \cite[1.5]{BarlakSzabo16ss}, the dynamical system on the central sequence algebra $(F_{\infty,\alpha}(A),\tilde{\alpha}_\infty)$ does invariant under cocycle conjugacy or stabilization with the compacts. 
Thus we may without loss of generality replace $(\alpha,u)$ by $(\alpha\otimes\delta,u\otimes\eins)$ and show the claim in this case.
By the Packer-Raeburn stabilization trick from \cite[3.4]{PackerRaeburn89}, the 2-cocycle $u\otimes\eins$ is a coboundary, and so we may assume that $\alpha$ is a genuine action.

Let $\eps_n>0$ be a decreasing null sequence, and let $F_n^A\fin A$ be increasing finite subsets with dense union. Let $G'\subset G$ be a countable, dense subgroup. Let $D\subset\CD$ be a countable, dense, $\gamma|_{G'}$-invariant $\IQ[i]$-$*$-subalgebra. Let $F_n^D\fin D$ be increasing finite subsets with $D=\bigcup_{n\in\IN} F_n$, and let $K_n\subset G$ be an increasing sequence of compact sets with $G=\bigcup_{n\in\IN} K_n$. 

Fix some $n\in\IN$. Then, as the subgroups $G_k\subset G$ are open, there exists some $N\in\IN$ with $K_n\subset G_N$. Using \ref{rmk:cont-reps}, we find a commutative diagram 
\[
\xymatrix{
 && \ell^\infty_{\alpha|_{G_N}}(\IN,A) \ar[d] \\
 && \big( A_{\infty,\alpha|_{G_N}}\cap A' , \alpha_\infty|_{G_N} \big) \ar[d] \\
(\CD,\gamma|_{G_N}) \ar[uurr]^{\kappa=(\kappa_l)_l} \ar[urr]^{\psi} \ar[rr]^{\psi_0} && \big(F_{\infty,\alpha|_{G_N}}(A),\tilde{\alpha}_\infty|_{G_N}\big)
}
\]
where $\psi_0$ is a $G_N$-equivariant $*$-homomorphism, $\psi$ is a $G_N$-equivariant c.p.c.\ order zero map and $\kappa$ is a (not necessarily equivariant) linear map.

As $\kappa$ is a lift for both $\psi$ and $\psi_0$, we observe the following properties for all $x,y\in\CD$, $a\in A$ and compact sets $K\subset G_N$:
\begin{itemize}
\item $\dst\limsup_{l\to\infty} \|\kappa_l(x)\|\leq\|x\|$;
\item $\dst\lim_{l\to\infty} \|[\kappa_l(x),a]\|=0$;
\item $\dst\lim_{l\to\infty} \|\kappa_l(\eins)a-a\|=0$;
\item $\dst\lim_{l\to\infty} \|\kappa_l(xy)\kappa_l(\eins)-\kappa_l(x)\kappa_l(y)\|=0$;
\item $\dst\lim_{l\to\infty}~ \max_{g\in K}~ \|(\alpha_g\circ\kappa_l)(x)-(\kappa_l\circ\gamma_g)(x)\|=0$.
\end{itemize}
In particular, we find $l(n)\in\IN$ such that the following are satisfied for all $x,y\in F_n^D$ and $a\in F_n^A$:
\begin{itemize}
\item $\|\kappa_{l(n)}(x)\|\leq\|x\|+\eps_n$;
\item $\|[\kappa_{l(n)}(x),a]\|\leq\eps_n$;
\item $\|\kappa_{l(n)}(\eins)a-a\|\leq\eps_n$;
\item $\|\kappa_{l(n)}(xy)\kappa_{l(n)}(\eins)-\kappa_{l(n)}(x)\kappa_{l(n)}(y)\|\leq\eps_n$;
\item $\dst \max_{g\in K_n}~ \|(\alpha_g\circ\kappa_{l(n)})(x)-(\kappa_{l(n)}\circ\gamma_g)(x)\|\leq\eps_n$.
\end{itemize}
By these properties, the $\IQ[i]$-$*$-linear map $\phi=[(\kappa_{l(n)})_n]: D\to A_{\infty,\alpha}$ is well-defined, contractive and satisfies
\begin{itemize}
\item $[\phi(x),a]=0$ for all $x\in D$ and $a\in A$;
\item $\phi(\eins)a=a$ for all $a\in A$;
\item $\phi(xy)\phi(\eins)=\phi(x)\phi(y)$ for all $x,y\in D$;
\item $\alpha_g\circ\phi = \phi\circ\gamma_g$ for all $g\in G'$.
\end{itemize}
Thus this map extends continuously to an equivariant c.p.c.\ order zero map $\phi: (\CD,\gamma)\to (A_{\infty,\alpha}\cap A',\alpha_{\infty})$ with $\phi(\eins)a=a$ for all $a\in A$. Then $\phi_0=\phi+\ann(A,A_{\infty,\alpha})$ yields the desired equivariant and unital $*$-homomorphism from $(\CD,\gamma)$ to $\big(F_{\infty,\alpha}(A),\tilde{\alpha}_\infty\big)$.
\end{proof}

\begin{theorem} \label{thm:reduction-subgroups}
Let $G$ be a second-countable, locally compact group. Let $A$ be a separable \cstar-algebra with a cocycle action $(\alpha,u): G\curvearrowright A$. Let $\CD$ be a separable, unital \cstar-algebra and $\gamma: G\curvearrowright\CD$ an action. Let $\set{G_n}_{n\in\IN}$ be an increasing family of open subgroups of $G$ such that $G=\bigcup_{n\in\IN} G_n$.
\begin{enumerate}[label=\textup{(\roman*)},leftmargin=*] 
\item The action $\gamma$ is semi-strongly self-absorbing if and only if for every $n\in\IN$, the restriction $\gamma|_{G_n}$ is semi-strongly self-absorbing. \label{reduction-subgroups:1}
\item Suppose that $\gamma$ is semi-strongly self-absorbing. Then $(\alpha,u)\cc(\alpha\otimes\gamma,u\otimes\eins_\CD)$ if and only if for every $n\in\IN$, one has $(\alpha, u)|_{G_n}\cc (\alpha\otimes\gamma, u\otimes\eins_\CD)|_{G_n}$. \label{reduction-subgroups:2}
\item Suppose that $\gamma$ is semi-strongly self-absorbing, and that $\beta: G\curvearrowright\CD$ is another action. Then $\beta\scc\gamma$ if and only if for every $n\in\IN$, one has $\beta|_{G_n}\scc\gamma|_{G_n}$. \label{reduction-subgroups:3}
\end{enumerate}
\end{theorem}
\begin{proof}
The implication ``$\Rightarrow$'' is clear in every statement, so let us show the ``$\Leftarrow$'' implication everywhere.

\ref{reduction-subgroups:1}: Suppose that $\gamma|_{G_n}$ is semi-strongly self-absorbing for every $n\in\IN$. 
Then for every $n\in\IN$, we see by \ref{thm:sssa} that the action $\gamma|_{G_n}$ has approximately $G_n$-inner half-flip and there exists a unital and equivariant $*$-homomorphism from $(\CD,\gamma|_{G_n})$ to $(\CD_{\infty,\gamma|_{G_n}}\cap\CD',\gamma_\infty|_{G_n})$. 
By \ref{cor:half-flip-subgroup}, it follows that $\gamma$ has approximately $G$-inner half-flip. 
By \ref{lemma:F(A)-subgroup}, it follows that there exists a unital and equivariant $*$-homomorphism from $(\CD,\gamma)$ to $(\CD_{\infty,\gamma}\cap\CD',\gamma_\infty)$. Thus $\gamma$ is semi-strongly self-absorbing by \ref{thm:sssa}.

\ref{reduction-subgroups:2}: Suppose that $(\alpha,u)|_{G_n}\cc (\alpha\otimes\gamma, u\otimes\eins)|_{G_n}$ for every $n\in\IN$.
By the equivariant McDuff Theorem \ref{equi-McDuff}, this means that for every $n\in\IN$, there exists a unital and equivariant $*$-homomorphism from $(\CD,\gamma|_{G_n})$ to $\big(F_{\infty,\alpha|_{G_n}}(A),\tilde{\alpha}_\infty|_{G_n}\big)$. 
By \ref{lemma:F(A)-subgroup}, there exists a unital and equivariant $*$-homomor\-phism from $(\CD,\gamma)$ to $\big(F_{\infty,\alpha}(A),\tilde{\alpha}_\infty\big)$. Thus $(\alpha,u)\cc (\alpha\otimes\gamma,u\otimes\eins)$ by the equivariant McDuff theorem.

\ref{reduction-subgroups:3}: Suppose that $\beta|_{G_n}\scc\gamma|_{G_n}$ for every $n\in\IN$. Then in particular, for every $n\in\IN$, the action $\beta|_{G_n}$ is semi-strongly self-absorbing, and the actions $\beta|_{G_n}$ and $\gamma|_{G_n}$ absorb each other. Thus the claim follows upon combining \ref{reduction-subgroups:1} and \ref{reduction-subgroups:2}.
\end{proof}


\section{Reducing $\CZ$-stable absorption to UHF-stable absorption}

\begin{rem} \label{rmk:Zpq}
Recall that for two mutually coprime supernatural numbers $p$ and $q$, one writes
\[
Z_{p,q} = \set{ f\in\CC\big( [0,1], M_p\otimes M_q \big) \mid f(0)\in M_p\otimes\eins, f(1)\in\eins\otimes M_q }.
\]
It has been shown in \cite[Section 3]{RordamWinter10} that, if $p$ and $q$ are of infinite type, then there exists a trace-collapsing unital $*$-monomorphism $\phi: Z_{p,q}\to Z_{p,q}$ such that the stationary inductive limit $\dst\lim_{\longrightarrow} \set{Z_{p,q},\phi}$ is isomorphic to the Jiang-Su algebra $\CZ$.
\end{rem}

For proving the main result of this section, we need to recall some previous results from \cite{Szabo16ssa2}.

\begin{prop}[\see{Szabo16ssa2}{2.6}] \label{prop:commutator-ue}
Let $G$ be a second-countable, locally compact group. Let $A$ be a unital \cstar-algebra and $\alpha: G\curvearrowright A$ an action. 
Let $\CD$ be a separable, unital \cstar-algebra and $\gamma: G\curvearrowright\CD$ a semi-strongly self-absorbing action. Assume $\alpha\cc\alpha\otimes\gamma$. 
Let $\phi_1,\phi_2: (\CD,\gamma)\to (A,\alpha)$
be two unital and equivariant $*$-homomorphisms. Then there exist sequences of unitaries $u_n,v_n\in\CU(A)$ satisfying
\[
\max_{g\in K} \Big( \|u_n-\alpha_g(u_n)\|+\|v_n-\alpha_g(v_n)\| \Big) \stackrel{n\to\infty}{\longrightarrow} 0
\]
for every compact set $K\subset G$ and
\[
\ad(u_nv_nu_n^*v_n^*)\circ\phi_1 \stackrel{n\to\infty}{\longrightarrow} \phi_2
\]
in point-norm.
\end{prop}

\begin{prop}[\see{Szabo16ssa2}{1.19}] \label{prop:Z-stable-commutators}
Let $G$ be a second-countable, locally compact group. Let $A$ be a unital \cstar-algebra and $\alpha: G\curvearrowright A$ an action. Assume $\alpha\cc\alpha\otimes\id_\CZ$.
Then $\alpha$ is unitarily regular. 
Moreover, for every separable, $\alpha_\infty$-invariant \cstar-subalgebra $B\subset A_{\infty,\alpha}$, the fixed-point algebra of the relative commutant $(A_{\infty,\alpha}\cap B')^{\alpha_\infty}$ is $K_1$-injective.
\end{prop}

\begin{theorem}[\see{Szabo16ssa2}{4.9}] \label{thm:extensions}
Let $G$ be a second-countable, locally compact group. Let $\gamma: G\curvearrowright\CD$ be a semi-strongly self-absorbing action. If $\gamma$ is unitarily regular, then the class of all separable, $\gamma$-absorbing $G$-\cstar-dynamical systems is closed under equivariant extensions.
\end{theorem}

Recall the following technical property of semi-strongly self-absorbing actions, which arises as a consequence from a basic homotopy Lemma proved in \cite{Szabo16ssa2}.
See also \cite[Section 4]{GardellaLupini16}, where compelling Model theoretic evidence is given for the fact that dynamical systems induced on relative commutants like below are virtually indistinguishable from the surrounding system. 

\begin{lemma}[\see{Szabo16ssa2}{2.14}] \label{lemma:basic-homotopy}
Let $G$ be a second-countable, locally compact group. Let $\CD$ be a separable, unital \cstar-algebra and $\gamma: G\curvearrowright\CD$ a semi-strongly self-absorbing action. Let $A$ be a unital \cstar-algebra and $\alpha: G\curvearrowright A$ an action with $\alpha\scc\alpha\otimes\gamma$. Let $\psi: (\CD,\gamma)\to (A_{\infty,\alpha}, \alpha_\infty)$ be a unital and equivariant $*$-homomorphism. Then
\[
\CU_0\Big( \big( A_{\infty,\alpha}\cap\psi(\CD)' \big)^{\alpha_\infty} \Big) = \CU_0\big( (A_{\infty,\alpha})^{\alpha_\infty} \big) \cap \psi(\CD)'.
\]
In other words, a unitary in the fixed-point algebra $\big( A_{\infty,\alpha}\cap\psi(\CD)' \big)^{\alpha_\infty}$ is homotopic to $\eins$ precisely when it is homotopic to $\eins$ inside the larger fixed-point algebra $(A_{\infty,\alpha})^{\alpha_\infty}$.
\end{lemma}

The following is the main result of this section:

\begin{theorem} \label{thm:reduction-Z}
Let $G$ be a second-countable, locally compact group. Let $\CD$ be a separable, unital \cstar-algebra and $\gamma: G\curvearrowright\CD$ an action. Let $A$ be a separable \cstar-algebra and $(\alpha,u): G\curvearrowright A$ a cocycle action. Let $p$ and $q$ be two mutually coprime supernatural numbers of infinite type. 
\begin{enumerate}[label=\textup{(\roman*)},leftmargin=*] 
\item The action $\gamma\otimes\id_\CZ$ is semi-strongly self-absorbing if and only if $\gamma\otimes\id_{\IU}$ is semi-strongly self-absorbing for $\IU\in\set{M_p, M_q}$. \label{reduction-Z:1}
\item Suppose that $\gamma$ is semi-strongly self-absorbing. 
Then one has $(\alpha\otimes\id_\CZ,u\otimes\eins_\CZ)\cc (\alpha\otimes\gamma\otimes\id_\CZ,u\otimes\eins_\CD\otimes\eins_\CZ)$ if and only if one has $(\alpha\otimes\id_\IU,u\otimes\eins_\IU)\cc (\alpha\otimes\gamma\otimes\id_\IU,u\otimes\eins_\CD\otimes\eins_\IU)$ for $\IU\in\set{M_p, M_q}$. \label{reduction-Z:2}
\end{enumerate}
\end{theorem}
\begin{proof}
The implication ``$\Rightarrow$'' is clear in every statement because of $\IU\cong\IU\otimes\CZ$, so let us show the ``$\Leftarrow$'' implication everywhere. We shall start with \ref{reduction-Z:2} and use it to prove \ref{reduction-Z:1}.

\ref{reduction-Z:2}: We may assume without loss of generality that $(\alpha,u)\cc (\alpha\otimes\id_\CZ,u\otimes\eins)$ and $\gamma\cc\gamma\otimes\id_\CZ$. Note that by \ref{prop:Z-stable-commutators} and \ref{thm:extensions}, this implies that separable, $\gamma$-absorbing $G$-\cstar-dynamical systems are closed under equivariant extensions. 

Consider the Hilbert space $\CH=\ell^2(\IN)\bar{\otimes} L^2(G)$ and let $\delta: G\curvearrowright\CK(\CH)$ be the unitarily implemented action that is induced by the left-regular representation of $G$ on $L^2(G)$. 
By the Packer-Raeburn stabilization trick from \cite[3.4]{PackerRaeburn89}, the 2-cocycle $u\otimes\eins$ with respect to $\alpha\otimes\delta$ is a coboundary, and thus $(\alpha\otimes\delta,u\otimes\eins)$ is exterior equivalent to a genuine action.
By \cite[4.30]{BarlakSzabo16ss}, the property of $\gamma$-absorption is invariant under equivariant Morita equivalence. In particular, we may replace $(\alpha,u)$ by a genuine action on $A\otimes\CK(\CH)$, or alternatively just assume that $u=\eins$.

Denote $I=\CC_0(0,1)\otimes M_p\otimes M_q$ and $Q=M_p\oplus M_q$. From the canonical extension of \cstar-algebras
\[
\xymatrix{
0 \ar[r]& I \ar[r]& Z_{p,q} \ar[r]& Q \ar[r]& 0,
}
\]
we get the equivariant extension
\[
\xymatrix@C-3mm{
0 \ar[r]& (A\otimes I,\alpha\otimes\id_I) \ar[r] &  (A\otimes Z_{p,q},\alpha\otimes\id_{Z_{p,q}}) \ar[r] &  (A\otimes Q, \alpha\otimes\id_Q) \ar[r] & 0.
}
\]
Since $\alpha\otimes\id_\IU\cc\alpha\otimes\gamma\otimes\id_\IU$ for $\IU\in\set{M_p, M_q}$ by assumption, it is clear that $\alpha\otimes\id_I\cc\alpha\otimes\gamma\otimes\id_I$ and $\alpha\otimes\id_Q\cc\alpha\otimes\gamma\otimes\id_Q$.
Hence also $\alpha\otimes\id_{Z_{p,q}}\cc\alpha\otimes\gamma\otimes\id_{Z_{p,q}}$ by virtue of this extension.
As the Jiang-Su algebra $\CZ$ arises as a stationary inductive limit of $Z_{p,q}$ (see \ref{rmk:Zpq}), we have
\[
(A\otimes\CZ,\alpha\otimes\id_\CZ) \cong \lim_{\longrightarrow}\ (A\otimes Z_{p,q}, \alpha\otimes\id_{Z_{p,q}} ).
\]
Since $\gamma$-absorption passes to equivariant inductive limits by \cite[1.10]{Szabo16ssa2}, this shows our claim.

\ref{reduction-Z:1}: Set $\IU_1=M_p$, $\IU_2=M_q$ and $\IW=\IU_1\otimes\IU_2$. Suppose that $\gamma\otimes\id_{\IU_i}$ is semi-strongly self-absorbing for $i=1,2$. We will need to go through two steps in order to prove that $\gamma\otimes\id_\CZ$ is semi-strongly self-absorbing.

{\bf Step 1:} The first and most difficult step is to show that $\gamma\otimes\id_\CZ$ has approximately $G$-inner flip.
For a unital \cstar-algebra $C$, denote by $\sigma^C\in\Aut(C\otimes C)$ the flip automorphism. 
Set $B=\CD\otimes\CD$, $\beta=\gamma\otimes\gamma$ and consider the $\beta$-equivariant automorphism $\sigma=\sigma^\CD\in\Aut(B,\beta)$ given by the flip. 
Then $\beta\otimes\id_{\IU_i}$ is semi-strongly self-absorbing for $i=1,2$. Hence by \ref{prop:commutator-ue}, we can find
\[
u_i, v_i \ \in \ \CU\Big( (B\otimes\IU_i)_{\infty,\beta\otimes\id}^{(\beta\otimes\id)_\infty} \Big),\quad i=1,2
\]
with $\ad(u_iv_iu_i^*v_i^*)(b\otimes c_i) = \sigma(b)\otimes c_i$ for $b\in B$ and $c_i\in\IU_i$.
We may naturally view $B\otimes\IU_i\subset B\otimes\IW$ for $i=1,2$, and thus define
\[
z = (u_1v_1u_1^*v_1^*)^*(u_2v_2u_2^*v_2^*) \ \in \ \CU\Big( \big( (B\otimes\IW)_{\infty,\beta\otimes\id}\cap (B\otimes\IW)' \big)^{(\beta\otimes\id)_\infty} \Big)
\]
By \ref{prop:Z-stable-commutators}, it follows that the unitary $z$ is homotopic to the unit inside $(B\otimes\IW)_{\infty,\beta\otimes\id}^{(\beta\otimes\id)_\infty}$. By the basic homotopy Lemma \ref{lemma:basic-homotopy}, we thus get that $z$ is homotopic to the unit by some unitary path (note the slight abuse of notation)
\[
z: [0,1]\to \CU\Big( \big( (B\otimes\IW)_{\infty,\beta\otimes\id}\cap (B\otimes\IW)' \big)^{(\beta\otimes\id)_\infty} \Big)
\]
with $z(0)=\eins$ and $z(1)=z$. Let us consider the unitary path
\[
w: [0,1] \to \CU\Big( (B\otimes\IW)_{\infty,\beta\otimes\id}^{(\beta\otimes\id)_\infty} \Big),\quad w(t)=(u_1v_1u_1^*v_1^*)z(t).
\]
We see that 
\[
w(0)=u_1v_1u_1^*v_1^* \ \in \ \CU\Big( (B\otimes\IU_1)_{\infty,\beta\otimes\id}^{(\beta\otimes\id)_\infty} \Big)
\]
and 
\[
w(1)=u_2v_2u_2^*v_2^* \ \in \ \CU\Big( (B\otimes\IU_2)_{\infty,\beta\otimes\id}^{(\beta\otimes\id)_\infty} \Big).
\]
Moreover, we have
\[
w(t)(b\otimes c_1\otimes c_2)w(t)^* = \sigma(b)\otimes c_1\otimes c_2\quad\text{for all } b\in B, c_i\in\IU_i.
\]
Thus we can view $w$ as a unitary
\[
w\in \CU\Big( (B\otimes Z_{p,q})_{\infty,\beta\otimes\id}^{(\beta\otimes\id)_\infty} \Big) \ \stackrel{\ref{rmk:Zpq}}{\subset} \ \CU\Big( (B\otimes\CZ)_{\infty,\beta\otimes\id}^{(\beta\otimes\id)_\infty} \Big)
\]
with the property $w(b\otimes\eins)w^*=\sigma(b)\otimes\eins$ for all $b\in B$. Note that the image of
\[
(\CZ,\id_\CZ) \to (B\otimes\CZ)_{\infty,\beta\otimes\id}^{(\beta\otimes\id)_\infty},\quad x\mapsto w(\eins\otimes x)w^*
\]
commutes with $B\otimes\eins$. 
Using the uniqueness result \ref{prop:commutator-ue} and a reindexation trick, this map is $G$-unitarily equivalent to the canonical map $x\mapsto\eins\otimes x$, where we view it as a map with codomain being the relative commutant of $B\otimes\eins$.
By perturbing $w$ with the resulting unitary in $(B\otimes\CZ)_{\infty,\beta\otimes\id}^{(\beta\otimes\id)_\infty}\cap (B\otimes\eins)'$, if necessary, we may thus assume that $w(b\otimes x)w^*=\sigma(b)\otimes x$ for all $b\in B$ and $x\in\CZ$. 

To summarize, all of this shows that 
\[
\sigma^\CD\otimes\id_\CZ\ue{G}\id_B\otimes\id_\CZ=\id_{\CD\otimes\CD}\otimes\id_\CZ.
\]
Combining this with the fact that $\CZ\cong\CZ\otimes\CZ$ has approximately inner flip, we see that
\[
\sigma^{\CD\otimes\CZ} = \sigma_{23}^{-1}\circ(\id_{\CD\otimes\CD}\otimes\sigma^\CZ)\circ(\sigma^\CD\otimes\id_{\CZ\otimes\CZ})\circ\sigma_{23} \ue{G} \id_{\CD\otimes\CZ\otimes\CD\otimes\CZ},
\]
where $\sigma_{23}$ denotes the isomorphism from $\CD\otimes\CZ\otimes\CD\otimes\CD$ to $\CD\otimes\CD\otimes\CZ\otimes\CZ$ flipping the second and third tensors.

{\bf Step 2:} Let us now show the claim. From the previous step, we know that $\gamma\otimes\id_\CZ$ has approximately $G$-inner flip. Thus the infinite tensor power action
\[
(\gamma\otimes\id_\CZ)^{\otimes\infty} : G\curvearrowright (\CD\otimes\CZ)^{\otimes\infty}
\]
is strongly self-absorbing by \cite[3.3]{Szabo16ssa}. As $\IU_i\cong\IU_i\otimes\CZ$ for $i=1,2$, our assumptions imply
\[
\begin{array}{cll}
\gamma\otimes\id_{\IU_i} &\scc& (\gamma\otimes\id_{\IU_i})^{\otimes\infty} \\
&\cong& (\gamma\otimes\id_{\IU_i})^{\otimes\infty}\otimes(\gamma\otimes\id_{\CZ})^{\otimes\infty} \\
&\scc& \gamma\otimes\id_{\IU_i}\otimes (\gamma\otimes\id_{\CZ})^{\otimes\infty}
\end{array}
\]
for $i=1,2$.
By part \ref{reduction-Z:2} applied to $\gamma$ in place of $\alpha$ and $(\gamma\otimes\id_\CZ)^{\otimes\infty}$ in place of $\gamma$, it follows that
\[
\gamma\otimes\id_{\CZ} \scc \gamma\otimes\id_{\CZ}\otimes (\gamma\otimes\id_\CZ)^{\otimes\infty} \cong (\gamma\otimes\id_\CZ)^{\otimes\infty}.
\]
Using \ref{thm:sssa} this shows that $\gamma\otimes\id_\CZ$ is semi-strongly self-absorbing. 
\end{proof}


\section{Application to actions on strongly self-absorbing \cstar-algebras}

In this section, we shall obtain our main application of the results from the previous sections. First, we need to recall some results from the literature.

\begin{reme}
An automorphism $\alpha\in\Aut(A)$ on a unital \cstar-algebra $A$ is called strongly outer, if it is outer and if for every $\alpha$-invariant tracial state $\tau\in T(A)$, the induced automorphism of $\alpha$ on the weak closure $\pi_\tau(A)''$ is outer.
(Unlike in other sources from the literature, we shall not assume $T(A)\neq\emptyset$ for this definition. If $T(A)=\emptyset$, then strongly outer just means outer by convention.)
If $G$ is a discrete group, then a cocycle action $(\alpha,u): G\curvearrowright A$ is called pointwise strongly outer, if $\alpha_g$ is a strongly outer automorphism for every $g\in G\setminus\set{1_G}$.
\end{reme}

The following is a combination of results proved in \cite{Matui08, Matui11, IzumiMatui10} due to Matui and Izumi-Matui.

\begin{theorem} \label{thm:previous-uniqueness}
Let $\CD$ be a strongly self-absorbing \cstar-algebra satisfying the UCT that is not isomorphic to the Jiang-Su algebra. Let $d\geq 1$ be a number. Then any two pointwise strongly outer $\IZ^d$-actions on $\CD$ are strongly cocycle conjugate. Moreover, any such action is semi-strongly self-absorbing.
\end{theorem}
\begin{proof}
Note that by \cite[6.7]{TikuisisWhiteWinter15}, $\CD$ must be isomorphic to either a UHF algebra of infinite type, one of the Cuntz algebras $\CO_2$ or $\CO_\infty$, or tensor products between these.
By applying either one of the classification results \cite[5.4]{Matui11} of Matui, \cite[5.2]{Matui08} of Matui or \cite[6.18, 6.20]{IzumiMatui10} of Izumi-Matui, it follows that any two pointwise strongly outer $\IZ^d$-actions on $\CD$ are strongly cocycle conjugate. 
It also follows from \cite[5.9, 5.12]{Szabo16ssa} that such actions are automatically semi-strongly self-absorbing. 
\end{proof}

Here comes the main result of the paper:

\begin{theorem} \label{thm:main-result}
Let $\CD$ be a strongly self-absorbing \cstar-algebra satisfying the UCT. Let $G$ be a countable, torsion-free abelian group. Then any two pointwise strongly outer $G$-actions on $\CD$ are very strongly cocycle conjugate. Moreover, any such action is semi-strongly self-absorbing.
\end{theorem}
\begin{proof}
Let $\gamma_1, \gamma_2: G\curvearrowright\CD$ be two pointwise strongly outer actions.  We may write $G=\bigcup_{n\in\IN} G_n$ for an increasing sequence of finitely generated subgroups.
As $G$ is torsion-free and abelian, this implies in particular that for every $n$, the group $G_n$ is isomorphic to $\IZ^{d_n}$ for some $d_n\in\IN$. Let $p$ and $q$ be two mutually coprime supernatural numbers of infinite type. Set $\IU_1=M_p$ and $\IU_2=M_q$.

Then $\CD\otimes\IU_j$ is a strongly self-absorbing \cstar-algebras satisfying the UCT that is not isomorphic to the Jiang-Su algebra for $j=1,2$. 
Thus \ref{thm:previous-uniqueness} applies and we see that for every $n$, the $G_n$-action $(\gamma_i\otimes\id_{\IU_j})|_{G_n}$ is a semi-strongly self-absorbing action for $i=1,2$ and $j=1,2$, with
\[
(\gamma_1\otimes\id_{\IU_j})|_{G_n} \scc (\gamma_2\otimes\id_{\IU_j})|_{G_n} \scc (\gamma_1\otimes\gamma_2\otimes\id_{\IU_j})|_{G_n} \ ,\quad j=1,2.
\]
Thus we can apply \ref{thm:reduction-Z} to deduce that for every $n$, the $G_n$-action $(\gamma_i\otimes\id_\CZ)|_{G_n}$ is semi-strongly self-absorbing for $i=1,2$, with
\[
(\gamma_1\otimes\id_\CZ)|_{G_n} \scc (\gamma_2\otimes\id_\CZ)|_{G_n} \scc (\gamma_1\otimes\gamma_2\otimes\id_\CZ)|_{G_n}.
\]
As $n$ was arbitrary, it follows from \ref{thm:reduction-subgroups} that the actions $\gamma_1\otimes\id_\CZ$ and $\gamma_2\otimes\id_\CZ$ are semi-strongly self-absorbing and are strongly cocycle conjugate. 
Now one has $\gamma_1\scc\gamma_1\otimes\id_\CZ$ and $\gamma_2\scc\gamma_2\otimes\id_\CZ$ due to a result \cite[4.11]{MatuiSato14} of Matui-Sato. 
So $\gamma_1$ and $\gamma_2$ are equivariantly $\CZ$-stable and in particular unitarily regular by \ref{prop:Z-stable-commutators}. 
Since they absorb each other tensorially, it follows from \ref{thm:strong-absorption} that in fact $\gamma_1\vscc\gamma_1\otimes\gamma_2\vscc\gamma_2$. This finishes the proof.
\end{proof}

\begin{rem}
The strategy of the proof of \ref{thm:main-result} in order to obtain uniqueness results for actions on the Jiang-Su algebra, making crucial use of \ref{thm:reduction-Z}, should have more applications in the future because it relies on a general principle not depending on the acting group.
Note that the results from Section 5 in particular allow one to bypass having to solve some hard problems related to the vanishing of general cocycles, which has been considered by Matui-Sato in \cite{MatuiSato12_2, MatuiSato14} to show uniqueness for $\IZ^2$-actions on $\CZ$, and to show uniqueness for actions of the Klein bottle group $\IZ\rtimes_{-1}\IZ$ on $\CZ$. 
What is however seemingly inaccessible with our approach at the moment is to determine under what conditions a cocycle action on a strongly self-absorbing \cstar-algebra is cocycle conjugate to a genuine action; this could also be successfully tackled in Matui-Sato's approach \cite{MatuiSato12_2, MatuiSato14}.
\end{rem}


\bibliographystyle{gabor}
\bibliography{master}

\end{document}